\documentclass[10pt]{article}

\usepackage[margin=0.85in]{geometry}
\usepackage{amsmath, amssymb, amsfonts, amsthm}
\usepackage{bm}
\usepackage{natbib}
\usepackage{graphicx}
\usepackage{booktabs}
\usepackage{setspace}
\usepackage{hyperref}
\usepackage{xcolor}
\usepackage{subfig}

\newtheorem{definition}{Definition}
\newtheorem{theorem}{Theorem}
\newtheorem{proposition}{Proposition}

\title{ \null\vspace{-3cm} \textsf{Identifiability and Information-Based Inference for Epidemic Transmission Models Under Partial Observation}}
\vspace{2cm}
\author{Md Asaduzzaman\\Department of Engineering, University of Staffordshire}
\date{}

\begin{document}
\maketitle

\begin{abstract}
Inference for epidemic transmission on dynamic networks is fundamentally limited by latent infection times, incomplete contact histories, imperfect observation, and external sources of infection. Although coherent likelihood formulations are available for partially observed epidemic processes, considerably less is known about the theoretical limits of statistical inference under such observation mechanisms. This paper develops a unified framework for studying identifiability and Fisher information in epidemic transmission models observed on dynamic contact networks. We establish conditions for structural and local identifiability, derive observed and complete-data information matrices, and quantify information loss arising from unobserved transmission events and missing network information through a missing-information decomposition. We further investigate how observation frequency, network coverage, and measurement accuracy influence parameter estimability and statistical efficiency, providing a principled basis for evaluating surveillance strategies. Simulation studies demonstrate that the proposed framework accurately characterises the relationship between observation design, statistical information, and parameter estimation, with theoretical predictions closely matching finite-sample performance. The proposed framework clarifies the relationship between observation design, identifiability, and inferential precision, and provides a theoretical foundation for statistical inference in partially observed epidemic transmission models.
\end{abstract}


\linespread{1}
\section{Introduction}
Epidemic transmission on contact networks provides a natural setting in which disease spread is governed not only by biological progression, but also by the evolving pattern of interactions among individuals \citep{danon2011networks, bansal2010dynamic}. In contrast to classical compartmental models that assume homogeneous mixing, network-based formulations allow transmission to depend on who contacts whom and when, thereby capturing the role of local exposure, clustering, and temporal variation in social structure \citep{vecherin2026infection}. Such models are increasingly relevant for modern epidemic data, where contact histories may be observed from digital tracing, proximity sensors, or survey-based designs, but are rarely complete \cite{pellis2015eight}. Likelihood-based approaches are particularly attractive in this setting because they provide a coherent foundation for estimation, uncertainty quantification, and prediction \citep{breto2018modeling, tonsing2018profile, bu2022likelihood, abed2026spatial}. However, the inferential advantages of likelihood methods depend critically on the identifiability of the underlying parameters from the available observations \citep{bu2022likelihood, lam2022practical}.

A central difficulty in epidemic-network analysis is that the relevant data are typically only partially observed \citep{black2019importance, bu2022likelihood, kamkumo2025estimating}. Infection times are often inferred indirectly from symptoms, testing, or serology; recovery times may be recorded only coarsely; and contact networks are usually reconstructed from noisy or intermittent measurements \citep{eames2015six}. In addition, infections may arise from external sources outside the observed population, creating further confounding between internal transmission and outside exposure. These limitations are not merely computational inconveniences: they determine which parameters can be learned from the data at all. As a result, identifiability should be treated as a primary theoretical property of the model, rather than as an afterthought to be checked only after fitting \citep{lam2022practical, preston2025think}.

Existing work on epidemic inference over networks has made substantial progress in constructing stochastic and likelihood-based models for dynamic contact processes, including Bayesian data augmentation methods for partially observed outbreaks and continuous-time Markov formulations for epidemics on adaptive networks  \citep{bu2022likelihood, fintzi2017efficient, huang2024detecting}. More recent work has extended these ideas to heterogeneous transmission and stochastic EM algorithms under partial observation \citep{yang2012hybrid, bu2025stochastic}. At the same time, the broader identifiability literature has shown that even well-specified epidemic models can become non-identifiable once only aggregated or incomplete observations are available \citep{browning2022efficient, saucedo2024comparative}. What is still lacking is a unified framework that connects identifiability, information loss, and partial observation in epidemic models on dynamic networks.

This paper develops structural and local identifiability in terms of the distribution of the observed data. This perspective makes explicit that identifiability is a property of the observation mechanism as well as of the latent process. We then derive complete-data and observed-data information representations, using the missing-information principle to quantify how much inferential content is lost when infection times, contact histories, or external exposure pathways are unobserved. This decomposition is especially useful because it separates lack of information due to latent epidemiological events from lack of information due to the network itself being only partially recorded.

The resulting theory has two practical implications. First, it clarifies when transmission, progression, and external infection parameters are estimable from the observed data, and when they are only weakly identifiable or not identifiable at all. Second, it provides guidance for study design by showing how observation frequency, contact sampling, and reporting accuracy affect the amount of information available for inference. In this sense, the paper contributes both to the theory of interacting stochastic processes and to the design of epidemiological studies on evolving networks.

The remainder of the paper is organised as follows. Section \ref{sec2:frame} introduces the epidemic--network observation framework. Section \ref{sec3:iden} defines structural and local identifiability. Section \ref{sec4:inf-struc} develops complete-data and observed-data Fisher information representations. Section \ref{sec5:inf-loss} studies information loss under partial observation. Section \ref{sec6:asymp} considers asymptotic identifiability under increasing population size and increasing observation frequency. Section \ref{sec7:sim} presents simulation studies, and Section \ref{sec8:con} discusses implications for epidemic study design.

\subsection{Existing Literature}
This paper builds on four related strands of literature: stochastic epidemic modelling, likelihood-based inference for partially observed epidemic processes, epidemic models on static and dynamic networks, and identifiability and information theory for incomplete-data models. Classical stochastic epidemic models provide the probabilistic foundation for much of this work \citep{daley1999epidemic}. Standard susceptible--infectious--removed and related models have been studied extensively as continuous-time Markov processes, with inference commonly formulated through infection and removal event histories \citep{hethcote2000mathematics}. Early statistical developments established likelihood and Bayesian approaches for epidemic data, while also recognising that real epidemics are rarely fully observed \citep{bu2022likelihood, reeves2024model, morsomme2025exact}. In particular, infection times are often latent, and observed data may consist only of removal times, symptom onset times, diagnostic results, or final outbreak sizes \citep{becker1999statistical, morsomme2025exact, kamkumo2025estimating}. The work of \citet{o1999bayesian} is important in this respect, showing how Markov chain Monte Carlo methods can infer missing infection histories jointly with epidemic parameters in partially observed stochastic epidemics.

A second relevant strand concerns likelihood inference and missing-data information for stochastic processes. Event-history and counting-process methods provide a natural language for epidemic models because infection, progression, recovery, and contact events can be represented through state-dependent intensities. This connects epidemic inference to survival analysis and continuous-time Markov process theory. In incomplete-data settings, observed likelihoods are obtained by integrating over unobserved event histories, which is often computationally and theoretically challenging. The missing-information identity of \citet{louis1982finding} is central here because it expresses observed information as complete-data information minus the information lost through unobserved components. This identity provides a natural basis for quantifying the inferential cost of latent infection times and missing contact histories.

The third strand concerns epidemics on networks. Network epidemic models generalise homogeneous-mixing assumptions by allowing transmission to occur along edges of a contact graph. Reviews and monographs such as \citet{pastor2015epidemic} and \citet{istvan2019mathematics} show how network structure affects epidemic thresholds, final sizes, and intervention effects. Dynamic network models go further by allowing contacts to change during an outbreak. \citet{volz2007susceptible} formulated susceptible--infectious--removed epidemics on dynamic contact networks and showed how changing contacts alter epidemic dynamics. Such models are especially relevant for infectious disease studies in which contact behaviour changes because of symptoms, isolation, intervention, or risk perception.

A fourth strand studies inference when the contact network is unknown or partially observed. In this setting, epidemic data contain information about both the transmission process and the latent contact structure. \citet{groendyke2011bayesian} developed Bayesian methods for inferring contact-network parameters from epidemic data, illustrating the feasibility of joint epidemic-network inference. Related work has shown that uncertainty in the network can affect estimates of transmission parameters and uncertainty intervals \citep{bu2022likelihood, reeves2024model, bu2025stochastic, morsomme2025exact}. More recent approaches have begun to model imperfectly observed contact networks directly \citep{danon2011networks, heesterbeek2015modeling, almutiry2021contact}. However, a general framework for determining when transmission parameters are identifiable, and how much information is lost through partial network observation, remains limited.

The present work complements recent advances in likelihood-based inference for partially observed epidemic processes \citep{asaduzzaman2026complete} by establishing the theoretical foundations of identifiability and statistical information under incomplete observation. It addresses this gap by developing an identifiability and information framework for epidemic transmission models on partially observed dynamic networks. The contribution is not a new disease-specific epidemic model, but a statistical framework for understanding what can be learned from incomplete surveillance data. By defining identifiability through the observed-data law, decomposing complete and observed Fisher information, and characterising information loss due to latent infection times, missing contacts, and measurement error, the paper connects epidemic likelihood inference with practical questions of surveillance design. This perspective is well suited to biostatistical applications in which incomplete observation is unavoidable and the reliability of parameter interpretation is central.

\subsection{Contributions}
The paper makes the following contributions.
\begin{enumerate}
\item We define structural and local identifiability through the observed-data law, thereby making identifiability explicitly dependent on the surveillance design rather than only on the latent epidemic model.
\item We characterise how observation frequency, network coverage, measurement accuracy, and external infection pressure affect weak identifiability, parameter confounding, and the conditioning of the observed information matrix.
\item We provide simulation studies under high, moderate, and sparse observation regimes to demonstrate how information loss affects bias, RMSE, coverage, interval width, and empirical identifiability boundaries.
\end{enumerate}

\section{Background on Epidemic--Network Observation Framework}
\label{sec2:frame}
We consier a classical continuous-time susceptible--exposed--infectious--removed (SEIR) model embedded within a dynamic contact network as in \citep{asaduzzaman2026complete}. At any time \(t\), each individual occupies one of four epidemiological states: susceptible (\(S\)), exposed but not yet infectious (\(E\)), infectious (\(I\)), or removed (\(R\)) through recovery, isolation, or death. Disease transmission occurs when a susceptible individual has effective contact with an infectious individual. The parameter \(\beta\) denotes the transmission rate per effective contact, while \(\eta_{ab}\) represents the contact intensity between individuals in epidemiological states \(a\) and \(b\), thereby characterising the dynamic network structure through which infection spreads. Consequently, the infection hazard is jointly determined by the biological transmissibility of the pathogen and the underlying contact network. Following infection, individuals enter the exposed state, progress to the infectious state at rate \(\kappa\), and subsequently transition to the removed state at rate \(\gamma\). The resulting latent SEIR--network process defines the complete epidemic trajectory, which forms the underlying stochastic process from which the partially observed surveillance data are generated.

\begin{figure}[h]
\centering
\includegraphics[width=0.75\textwidth]{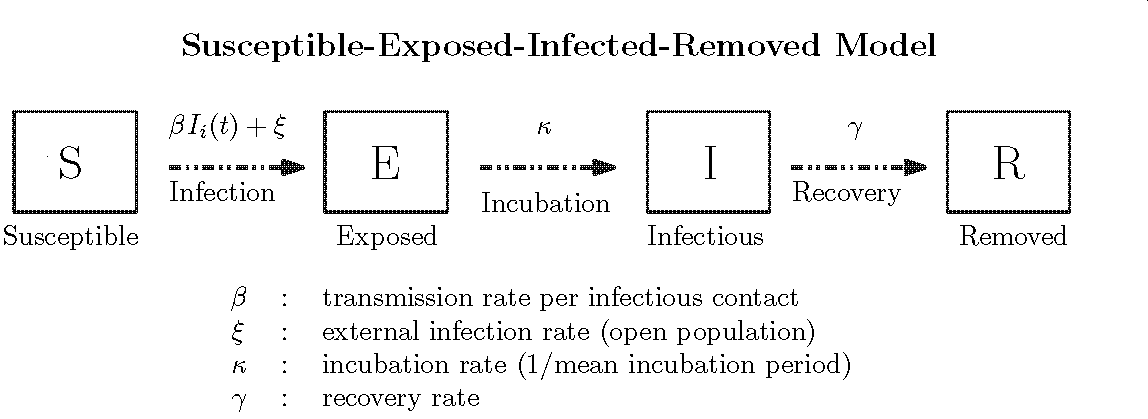}
\end{figure}

Let $\mathcal{V}_N=\{1,\ldots,N\}$ denote a fixed population observed over the time interval $[0,T]$. For each $t\in[0,T]$, define the latent state
\[
Z(t)=\big(X(t),A(t)\big),
\]
where $X(t)=(X_1(t),\ldots,X_N(t))$ is the vector of individual epidemic states and
\[
A(t)=\{A_{ij}(t):1\le i<j\le N\}
\]
is an undirected dynamic contact network. The state space is
\[
\mathcal{Z}=\mathcal{X}^N\times\mathcal{A}_N,
\qquad
\mathcal{X}=\{S,E,I,R\},
\]
where $\mathcal{A}_N=\{0,1\}^{N(N-1)/2}$ is the space of simple undirected graphs on $\mathcal{V}_N$.

We assume that $\{Z(t):0\le t\le T\}$ is a continuous-time Markov process with transition intensities determined by epidemic progression and network evolution. The full parameter vector is written as
\[
\theta=(\theta_{\mathrm{epi}},\theta_{\mathrm{net}},\theta_{\mathrm{obs}}),
\]
where
\[
\theta_{\mathrm{epi}}=(\beta,\xi,\kappa,\gamma),\qquad
\theta_{\mathrm{net}}=\{\eta_{ab},\tau_{ab}:a,b\in\mathcal{X}\},
\]
and $\theta_{\mathrm{obs}}$ contains the observation-process parameters.

\subsection{Latent Epidemic Process}

For each individual $i\in\mathcal{V}_N$, the disease state satisfies
\[
X_i(t)\in\{S,E,I,R\}.
\]
Let
\[
\mathcal{I}(t)=\{j:X_j(t)=I\}
\]
denote the set of infectious individuals at time $t$. The instantaneous infectious exposure experienced by susceptible individual $i$ is
\[
C_i(t)=\sum_{j\neq i}A_{ij}(t)\mathbf{1}\{X_j(t)=I\}.
\]
Conditional on the current latent state $Z(t)$, the disease transition intensities are
\[
q_i^{SE}\{Z(t)\}
=
\mathbf{1}\{X_i(t)=S\}\{\beta C_i(t)+\xi\},
\]
\[
q_i^{EI}\{Z(t)\}
=
\mathbf{1}\{X_i(t)=E\}\kappa,
\qquad
q_i^{IR}\{Z(t)\}
=
\mathbf{1}\{X_i(t)=I\}\gamma.
\]
Here $\beta$ is the per-contact transmission rate, $\xi$ is the external infection rate, $\kappa$ is the exposed-to-infectious progression rate, and $\gamma$ is the recovery rate.

The aggregate epidemic intensity is therefore
\[
\Lambda_{\mathrm{epi}}\{Z(t)\}
=
\sum_{i=1}^N
\left[
q_i^{SE}\{Z(t)\}
+
q_i^{EI}\{Z(t)\}
+
q_i^{IR}\{Z(t)\}
\right].
\]

\subsection{Dynamic Network Process}

The network process evolves through dyad-level link formation and dissolution. For each unordered pair $(i,j)$, $i<j$, the contact indicator satisfies
\[
A_{ij}(t)\in\{0,1\}.
\]
Conditional on the current epidemic states
\[
(X_i(t),X_j(t))=(a,b),
\]
the transition intensities of the dyad process are
\[
q_{ij}^{0\to1}\{Z(t)\}
=
\mathbf{1}\{A_{ij}(t)=0\}\eta_{ab},
\]
and
\[
q_{ij}^{1\to0}\{Z(t)\}
=
\mathbf{1}\{A_{ij}(t)=1\}\tau_{ab}.
\]
We assume symmetry,
\[
\eta_{ab}=\eta_{ba},
\qquad
\tau_{ab}=\tau_{ba},
\]
so that the network is undirected. The aggregate network intensity is
\[
\Lambda_{\mathrm{net}}\{Z(t)\}
=
\sum_{1\le i<j\le N}
\left[
q_{ij}^{0\to1}\{Z(t)\}
+
q_{ij}^{1\to0}\{Z(t)\}
\right].
\]

The total latent-process intensity is
\[
\Lambda\{Z(t)\}
=
\Lambda_{\mathrm{epi}}\{Z(t)\}
+
\Lambda_{\mathrm{net}}\{Z(t)\}.
\]

\subsection{Observation Process}

Let observations be collected at times
\[
0<t_1<\cdots<t_m\le T.
\]
The observed data are denoted by
\[
O=\{Y(t_r),B(t_r):r=1,\ldots,m\},
\]
where $Y(t_r)$ represents disease-related observations and $B(t_r)$ represents contact-network observations.

The observation model is conditionally independent given the latent process:
\[
p_\theta(O\mid Z)
=
\prod_{r=1}^m
p_{\theta_{\mathrm{obs}}}\{Y(t_r)\mid X(t_r)\}
p_{\theta_{\mathrm{obs}}}\{B(t_r)\mid A(t_r)\}.
\]

For symptom observations, let $Y_i(t_r)\in\{0,1\}$ denote whether symptoms are reported for individual $i$ at time $t_r$. A simple Bernoulli observation model is
\[
P\{Y_i(t_r)=1\mid X_i(t_r)=E\}=p_E,
\qquad
P\{Y_i(t_r)=1\mid X_i(t_r)=I\}=p_I,
\]
with corresponding probabilities for $S$ and $R$ specified according to the assumed symptom-reporting mechanism.

For contact observations, let $B_{ij}(t_r)\in\{0,1\}$ denote the observed contact status of dyad $(i,j)$ at time $t_r$. We assume a misclassification model
\[
P\{B_{ij}(t_r)=1\mid A_{ij}(t_r)=1\}=s,
\]
and
\[
P\{B_{ij}(t_r)=0\mid A_{ij}(t_r)=0\}=c,
\]
where $s$ and $c$ denote the sensitivity and specificity of contact observation.

This formulation separates the latent epidemic--network process from the observation process while preserving their probabilistic dependence. The resulting model is a partially observed continuous-time Markov process on a dynamic graph, and identifiability is therefore determined by both the latent transition structure and the information retained by the observation mechanism.

\section{Identifiability Framework}
\label{sec3:iden}
Let $\Theta\subseteq\mathbb{R}^p$ denote the parameter space and let
\[
\theta=
(\beta,\xi,\kappa,\gamma,\eta,\tau,p_E,p_I,s,c)\in\Theta
\]
collect the epidemic, network, and observation parameters. For each $\theta\in\Theta$, let $P_\theta^O$ denote the probability law induced on the observed data
\[
O=\{Y(t_r),B(t_r):r=1,\ldots,m\}.
\]
Equivalently, if $Z$ denotes the latent epidemic--network trajectory, then
\[
p_\theta(O)
=
\int p_\theta(O\mid Z)\,p_\theta(Z)\,dZ,
\]
where the integral is understood over the space of latent continuous-time paths. Thus, identifiability is defined through the observed-data law rather than through the complete-data process alone.

\subsection{Structural Identifiability}
\begin{definition}[Structural identifiability]
A parameter value $\theta\in\Theta$ is structurally identifiable from the observation process $O$ if
\[
P_\theta^O=P_{\theta'}^O
\quad\Longrightarrow\quad
\theta=\theta'
\]
for all $\theta'\in\Theta$.
\end{definition}

Equivalently, the parameter is structurally identifiable if the mapping
\[
\mathcal{M}_O:\Theta\to\mathcal{P}(\mathcal{O}),
\qquad
\mathcal{M}_O(\theta)=P_\theta^O,
\]
is injective, where $\mathcal{P}(\mathcal{O})$ denotes the set of probability measures on the observation space $\mathcal{O}$. This formulation makes explicit that identifiability depends jointly on the latent epidemic--network dynamics and the observation mechanism. A parameter may be identifiable under complete observation of $Z$ but fail to be identifiable under the reduced observation $O$.

\subsection{Local Identifiability}

Let
\[
\ell_O(\theta)=\log p_\theta(O)
\]
denote the observed-data log-likelihood. Suppose $\ell_O(\theta)$ is twice differentiable in a neighbourhood of $\theta_0$. Define the observed score vector
\[
S_O(\theta)=
\frac{\partial \ell_O(\theta)}{\partial\theta}
\]
and the observed Fisher information matrix
\[
I_O(\theta)=
E_\theta\left[
S_O(\theta)S_O(\theta)^\top
\right]
=
-
E_\theta\left[
\frac{\partial^2\ell_O(\theta)}
{\partial\theta\partial\theta^\top}
\right],
\]
under standard regularity conditions.

\begin{definition}[Local identifiability]
The parameter $\theta$ is locally identifiable at $\theta_0$ if there exists a neighbourhood $U(\theta_0)\subseteq\Theta$ such that
\[
P_\theta^O=P_{\theta_0}^O,
\qquad
\theta\in U(\theta_0),
\]
implies
\[
\theta=\theta_0.
\]
\end{definition}

A sufficient regularity condition for local identifiability is
\[
\operatorname{rank}\{I_O(\theta_0)\}=p.
\]
If $I_O(\theta_0)$ is rank deficient, then there exists at least one direction $v\neq0$ such that
\[
v^\top I_O(\theta_0)v=0,
\]
indicating that infinitesimal perturbations of $\theta_0$ in direction $v$ are not distinguishable from the observed data to second order.

\subsection{Observation Equivalence Classes}

Partial observation induces equivalence classes of parameter values. For a given $\theta\in\Theta$, define
\[
[\theta]_O
=
\{\theta'\in\Theta:P_{\theta'}^O=P_\theta^O\}.
\]
The model is structurally identifiable at $\theta$ if and only if
\[
[\theta]_O=\{\theta\}.
\]
If $[\theta]_O$ contains more than one element, the parameters in this class are observationally equivalent and cannot be distinguished by any procedure based solely on $O$.

In epidemic--network models, such equivalence classes arise naturally. For example, incomplete contact observation may induce confounding between the internal transmission rate $\beta$ and the external infection rate $\xi$. Similarly, latent incubation periods may induce dependence between the infection rate $\beta$ and the progression rate $\kappa$, while imperfect symptom reporting may confound disease-state parameters with observation parameters $(p_E,p_I)$. Thus, identifiability is determined not only by the mechanistic structure of the epidemic process but also by the information retained by the observation process.

\subsection{Parameter-Specific Identifiability}
In many applications, interest lies in a lower-dimensional component
\[
\psi=g(\theta),
\]
such as the transmission rate $\beta$, the external infection rate $\xi$, or a reproduction-related functional. We say that $\psi$ is identifiable from $O$ if
\[
P_\theta^O=P_{\theta'}^O
\quad\Longrightarrow\quad
g(\theta)=g(\theta').
\]
This distinction is important because the full parameter vector may not be identifiable even when certain scientific functionals are. For example, individual network formation and dissolution parameters may be weakly identified, while a lower-dimensional transmission functional may remain estimable under suitable observation regimes.

\subsection{Identifiability Under Observation Regimes}
Let $\mathcal{R}$ denote an observation regime specifying symptom-reporting frequency, contact-sampling frequency, and measurement accuracy. The induced observed-data law may be written as $P_{\theta,\mathcal{R}}^O$. Identifiability is therefore regime-dependent. A parameter value $\theta$ may be identifiable under a high-information regime $\mathcal{R}_1$ but not under a sparse regime $\mathcal{R}_2$. We write
\[
\mathcal{M}_{O,\mathcal{R}}(\theta)=P_{\theta,\mathcal{R}}^O
\]
to emphasise this dependence.

This observation motivates the information analysis in the following sections. By studying the rank and conditioning of $I_O(\theta)$ under different regimes $\mathcal{R}$, we can characterise when the available observations contain sufficient information to distinguish transmission, progression, network, and observation parameters.

\section{Information Structure of the Partially Observed Epidemic--Network Process}
\label{sec4:inf-struc}
The inferential properties of the model are determined by the amount of information retained in the observation process about the underlying epidemic--network trajectory. Since the latent process
\[
Z=\{X(t),A(t):0\le t\le T\}
\]
is only partially observed through
\[
O=\{Y(t_r),B(t_r):r=1,\ldots,m\},
\]
it is useful to distinguish between information available under complete observation of $Z$ and information available under the observed data alone. This distinction provides a natural framework for quantifying the effect of latent infection times, missing contacts, and observation error on parameter estimation.

\subsection{Complete-Data Fisher Information}

Let
\[
L_C(\theta)=p_\theta(Z,O)
\]
denote the complete-data likelihood and define the corresponding log-likelihood by
\[
\ell_C(\theta)=\log L_C(\theta).
\]

The complete-data score function is
\[
S_C(\theta)
=
\frac{\partial \ell_C(\theta)}
{\partial\theta},
\]
and the complete-data Fisher information matrix is
\[
I_C(\theta)
=
E_\theta
\left[
S_C(\theta)S_C(\theta)^\top
\right].
\]

Under standard regularity conditions,
\[
I_C(\theta)
=
-
E_\theta
\left[
\frac{\partial^2\ell_C(\theta)}
{\partial\theta\partial\theta^\top}
\right].
\]

The matrix $I_C(\theta)$ represents the maximum information available for inference if the epidemic trajectory, network evolution, and observation process were observed without error. It therefore serves as an upper bound on the information obtainable under any observation regime.

For continuous-time epidemic processes, $I_C(\theta)$ admits an event-history representation involving the compensators of epidemic and network transition counting processes. Consequently, complete-data information accumulates through observed infection, progression, recovery, link-formation, and link-dissolution events.

\subsection{Observed Fisher Information}

The observed-data likelihood is obtained by integrating over all latent epidemic--network trajectories:
\[
L_O(\theta)
=
p_\theta(O)
=
\int p_\theta(Z,O)\,dZ.
\]

The observed-data log-likelihood is
\[
\ell_O(\theta)
=
\log p_\theta(O),
\]
with corresponding score function
\[
S_O(\theta)
=
\frac{\partial \ell_O(\theta)}
{\partial\theta}.
\]

The observed Fisher information matrix is defined by
\[
I_O(\theta)
=
E_\theta
\left[
S_O(\theta)S_O(\theta)^\top
\right]
=
-
E_\theta
\left[
\frac{\partial^2\ell_O(\theta)}
{\partial\theta\partial\theta^\top}
\right].
\]

Unlike the complete-data information, $I_O(\theta)$ depends explicitly on the observation regime. Sparse symptom observations, infrequent network measurements, or low observation accuracy reduce the curvature of the observed likelihood and therefore weaken parameter identifiability. In particular, rank deficiencies in $I_O(\theta)$ correspond to locally non-identifiable directions in parameter space.

\subsection{Information Decomposition and Missing Information}

The relationship between complete-data and observed-data information follows from the missing-information principle. Define the conditional covariance matrix of the complete-data score given the observed data:
\[
I_M(\theta)
=
E_\theta
\left[
\mathrm{Var}_\theta
\left\{
S_C(\theta)\mid O
\right\}
\right].
\]

The matrix $I_M(\theta)$ measures the uncertainty arising from unobserved epidemic states, latent infection times, unobserved network trajectories, and other missing components of the process.

The following theorem provides a decomposition of information.

\begin{theorem}[Louis Identity]
Assume that differentiation and integration may be interchanged and that the relevant expectations exist. Then
\[
I_C(\theta)
=
I_O(\theta)
+
I_M(\theta).
\]
Equivalently,
\[
I_O(\theta)
=
I_C(\theta)-I_M(\theta).
\]
\end{theorem}

\begin{proof}
Using Fisher's identity,
\[
S_O(\theta)
=
E_\theta
\left[
S_C(\theta)\mid O
\right],
\]
and applying the law of total variance to the complete-data score,
\[
\mathrm{Var}_\theta\{S_C(\theta)\}
=
\mathrm{Var}_\theta
\left[
E_\theta\{S_C(\theta)\mid O\}
\right]
+
E_\theta
\left[
\mathrm{Var}_\theta
\{S_C(\theta)\mid O\}
\right].
\]
Recognising the first term as $I_O(\theta)$ and the second as $I_M(\theta)$ yields the result.
\end{proof}

The decomposition provides a quantitative measure of information loss induced by partial observation. In the limiting case of complete observation,
\[
I_M(\theta)=0,
\]
so that
\[
I_O(\theta)=I_C(\theta).
\]
Conversely, as observation becomes increasingly sparse or noisy, $I_M(\theta)$ grows and the observed information decreases. The matrix
\[
R(\theta)
=
I_O(\theta)I_C(\theta)^{-1}
\]
may therefore be interpreted as a relative information operator, describing the proportion of complete-data information retained by a given observation regime.

This decomposition forms the basis for the identifiability analysis developed in subsequent sections. By studying the eigenstructure of $I_O(\theta)$ and the contribution of $I_M(\theta)$ under different observation regimes, it becomes possible to determine which parameters are estimable, which are weakly identifiable, and which remain fundamentally confounded.

\section{Information Loss Under Partial Observation}
\label{sec5:inf-loss}
The decomposition
\[
I_C(\theta)=I_O(\theta)+I_M(\theta)
\]
provides a natural way to quantify the inferential effect of partial observation. In the epidemic--network setting, missing information arises from at least three sources: unobserved infection times, incomplete contact histories, and measurement error in the observation process. Each source affects different components of the parameter vector and may introduce distinct forms of confounding. This section characterises these effects in terms of score variability, exposure uncertainty, and rank deficiencies of the observed Fisher information.

\subsection{Missing Infection Times}

Let $T_i^{SE}$ denote the infection time of individual $i$, with $T_i^{SE}=\infty$ if individual $i$ is not infected during $[0,T]$. Under complete observation, the infection contribution to the likelihood contains the event intensity
\[
\lambda_i^{SE}(t)
=
\mathbf{1}\{X_i(t)=S\}
\{\beta C_i(t)+\xi\},
\]
evaluated at the realised infection time, together with the integrated hazard over the susceptible period. When $T_i^{SE}$ is unobserved, inference is based on the conditional distribution of $T_i^{SE}$ given the observed data $O$.

Let
\[
S_{C,\psi}(\theta)
=
\frac{\partial \ell_C(\theta)}{\partial \psi},
\qquad
\psi=(\beta,\xi,\kappa)^\top ,
\]
denote the complete-data score for the transmission and incubation parameters. The missing information associated with latent infection times is
\[
I_M^{(T)}(\psi)
=
E_\theta\left[
\mathrm{Var}_\theta
\left\{
S_{C,\psi}(\theta)
\mid O
\right\}
\right].
\]
This matrix is positive semidefinite and satisfies
\[
I_O(\psi)
=
I_C(\psi)-I_M^{(T)}(\psi)-I_M^{(\mathrm{other})}(\psi),
\]
where $I_M^{(\mathrm{other})}(\psi)$ denotes missing information arising from unobserved network and observation components.

Thus, latent infection times reduce information about $\beta$, $\xi$, and $\kappa$ by increasing the conditional variability of the complete-data score. In particular, if symptom onset is observed but infection time is not, then multiple infection histories may be compatible with the same observed symptoms, producing uncertainty in both the timing of transmission events and the duration of the exposed state.

\subsection{Missing Contact Histories}

Transmission information depends on the exposure process
\[
C_i(t)
=
\sum_{j\neq i}
A_{ij}(t)\mathbf{1}\{X_j(t)=I\}.
\]
For the transmission rate $\beta$, the complete-data score contains terms of the form
\[
S_{C,\beta}(\theta)
=
\sum_{i:T_i^{SE}\le T}
\frac{C_i(T_i^{SE})}
{\beta C_i(T_i^{SE})+\xi}
-
\int_0^T
\mathbf{1}\{X_i(t)=S\}
C_i(t)\,dt .
\]
When $A(t)$ is unobserved or intermittently observed, the cumulative exposure
\[
H_i(T)
=
\int_0^T
\mathbf{1}\{X_i(t)=S\}C_i(t)\,dt
\]
is latent. The observed score for $\beta$ is then obtained by conditional averaging,
\[
S_{O,\beta}(\theta)
=
E_\theta
\left[
S_{C,\beta}(\theta)
\mid O
\right],
\]
and the missing information due to unobserved contacts is
\[
I_M^{(A)}(\beta)
=
E_\theta
\left[
\mathrm{Var}_\theta
\left\{
S_{C,\beta}(\theta)
\mid O
\right\}
\right].
\]

This expression shows explicitly that contact uncertainty affects transmission inference through uncertainty in the exposure process. If the observed contact data provide little information about $A(t)$, then different latent exposure histories may explain the same observed infection pattern, leading to weak identification of $\beta$.

\subsection{Sparse Observation Regimes}
Let $\mathcal{R}$ denote an observation regime, defined by symptom observation frequency, contact sampling frequency, and measurement accuracy. The observed Fisher information under regime $\mathcal{R}$ is denoted by
$I_O(\theta;\mathcal{R})$.

A natural measure of relative information is
\[
R(\theta;\mathcal{R})
=
I_C(\theta)^{-1/2}
I_O(\theta;\mathcal{R})
I_C(\theta)^{-1/2},
\]
whose eigenvalues lie in $[0,1]$ whenever $I_C(\theta)$ is positive definite. Small eigenvalues of $R(\theta;\mathcal{R})$ indicate directions in parameter space for which the observation regime retains little information.

We define an observation regime $\mathcal{R}$ to be weakly informative at $\theta$ if
\[
\lambda_{\min}\{I_O(\theta;\mathcal{R})\}
\approx 0,
\]
where $\lambda_{\min}(\cdot)$ denotes the smallest eigenvalue, or equivalently, if the condition number
\[
\phi_I(\theta;\mathcal{R})
=
\frac{\lambda_{\max}\{I_O(\theta;\mathcal{R})\}}
{\lambda_{\min}\{I_O(\theta;\mathcal{R})\}}
\]
is large. In such regimes, likelihood surfaces are nearly flat in at least one direction, and parameter estimates may be unstable even when numerical convergence appears satisfactory.

As symptom and contact observations become less frequent or less accurate, the missing-information matrix increases in the Loewner order,
\[
I_M(\theta;\mathcal{R}_2)
-
I_M(\theta;\mathcal{R}_1)
\succeq 0,
\]
for regimes $\mathcal{R}_2$ that are less informative than $\mathcal{R}_1$, provided the observation mechanisms are nested. Consequently,
\[
I_O(\theta;\mathcal{R}_2)
\preceq
I_O(\theta;\mathcal{R}_1),
\]
so sparse observation reduces or preserves, but cannot increase, the information available for estimation.

\subsection{External Infection Confounding}
External infection introduces a particularly important source of non-identifiability. The susceptible-to-exposed intensity can be written as
\[
\lambda_i^{SE}(t)
=
\mathbf{1}\{X_i(t)=S\}
\{\beta C_i(t)+\xi\}.
\]
The parameter $\beta$ governs infection through observed or latent contacts, whereas $\xi$ governs infection from outside the observed network. If the exposure process $C_i(t)$ is poorly observed, changes in $\beta$ may be offset by changes in $\xi$, producing similar observed epidemic trajectories.

\begin{theorem}[Transmission--external infection confounding]
Suppose that individual infection times and contact histories are not observed, and that the observation process depends on the latent epidemic--network process only through an aggregate exposure summary
\[
\bar C(t)
=
\frac{1}{N_S(t)}
\sum_{i:X_i(t)=S}
C_i(t),
\]
where $N_S(t)$ is the number of susceptible individuals at time $t$. If the observed-data law depends on $(\beta,\xi)$ only through
\[
\lambda_{\mathrm{agg}}(t)
=
\beta \bar C(t)+\xi,
\]
then $\beta$ and $\xi$ are not separately structurally identifiable unless $\bar C(t)$ varies over time in a way that is identifiable from $O$.
\end{theorem}

\begin{proof}
Suppose that two parameter values $(\beta,\xi)$ and $(\beta',\xi')$ satisfy
\[
\beta \bar C(t)+\xi
=
\beta'\bar C(t)+\xi'
\]
for all $t$ in the support of the observed process. Then the aggregate infection intensity is identical under the two parameter values. Since, by assumption, the observed-data law depends on $(\beta,\xi)$ only through $\lambda_{\mathrm{agg}}(t)$, it follows that
\[
P_{\beta,\xi}^O
=
P_{\beta',\xi'}^O .
\]
If $(\beta,\xi)\neq(\beta',\xi')$, the two parameter values are observationally equivalent. Hence $\beta$ and $\xi$ are not separately structurally identifiable.

Separate identifiability requires variation in $\bar C(t)$ that is itself recoverable from the observed data. In that case, equality of the aggregate hazard for sufficiently many distinct exposure values implies
\[
\beta=\beta',
\qquad
\xi=\xi',
\]
which breaks the equivalence class.
\end{proof}

This theorem formalises the intuition that transmission and external infection can be distinguished only when the observation process contains sufficient information about network exposure. Without such information, the observed epidemic curve may identify only an effective infection pressure, not its decomposition into internal and external components.

\section{Asymptotic Identifiability}
\label{sec6:asymp}
Structural identifiability is a finite-sample property of the statistical model, whereas practical estimability depends on the amount of information available in the observed data. In partially observed epidemic--network systems, the distinction is particularly important because parameters that are theoretically identifiable may remain weakly estimable under sparse observation. This section studies asymptotic regimes under which information accumulates and establishes conditions for asymptotic identifiability.

Let $I_O^{(N,\Delta)}(\theta)$ denote the observed Fisher information matrix corresponding to population size $N$ and observation interval $\Delta$. We are interested in the asymptotic behaviour of the eigenvalues of $I_O^{(N,\Delta)}(\theta)$ under increasing amounts of information.

\subsection{Increasing Population Size}

Consider a sequence of epidemic--network systems indexed by population size $N\rightarrow\infty$. Let $I_O^{(N)}(\theta)$ denote the observed information matrix under population size $N$. For local identifiability, it is necessary that the information matrix grows in all parameter directions.

A parameter vector $\theta$ is said to be asymptotically identifiable if
\[
\lambda_{\min}
\left\{
I_O^{(N)}(\theta)
\right\}
\rightarrow \infty
\qquad
(N\rightarrow\infty).
\]

Under standard regularity conditions, consistency of the maximum likelihood estimator requires
\[
N^{-1}I_O^{(N)}(\theta)
\rightarrow
J(\theta),
\]
where $J(\theta)$ is a positive-definite information matrix. Positive definiteness of $J(\theta)$ implies that all parameter directions accumulate information at rate $N$.

\begin{theorem}[Asymptotic identifiability under increasing population size]
Suppose that
\[
N^{-1}I_O^{(N)}(\theta)
\overset{p}{\longrightarrow}
J(\theta),
\]
where $J(\theta)$ is positive definite. Then
\[
\lambda_{\min}
\left\{
I_O^{(N)}(\theta)
\right\}
\rightarrow\infty,
\]
and $\theta$ is asymptotically identifiable.
\end{theorem}

\begin{proof}
Because
\[
N^{-1}I_O^{(N)}(\theta)\overset{p}{\to}J(\theta),
\]
and \(J(\theta)\) is positive definite, there exists \(k>0\) such that
\[
\lambda_{\min}\{J(\theta)\}=k.
\]
By continuity of the smallest eigenvalue,
\[
\lambda_{\min}\!\left\{N^{-1}I_O^{(N)}(\theta)\right\}\overset{p}{\to}k.
\]
Hence,
\[
\lambda_{\min}\!\left\{I_O^{(N)}(\theta)\right\}
=
N\lambda_{\min}\!\left\{N^{-1}I_O^{(N)}(\theta)\right\}
\overset{p}{\to}\infty.
\]
Therefore the observed information increases without bound in every direction, so the likelihood becomes increasingly locally curved around \(\theta\). This implies that no distinct parameter value can generate the same asymptotic likelihood behavior, and thus \(\theta\) is asymptotically identifiable.
\end{proof}

The result shows that identifiability emerges when information accumulates sufficiently rapidly relative to the dimension of the parameter space.

\subsection{Increasing Observation Frequency}

Let
\[
0=t_0<t_1<\cdots<t_m=T
\]
denote observation times with spacing
\[
\Delta=\max_r (t_r-t_{r-1}).
\]
As $\Delta\rightarrow0$, the observation process approaches continuous monitoring of the latent epidemic--network trajectory.

The missing-information matrix therefore depends on $\Delta$:
\[
I_M(\theta;\Delta).
\]

Since more frequent observations reveal more information about latent transitions,
\[
I_M(\theta;\Delta)
\rightarrow 0,
\qquad
\Delta\rightarrow0,
\]
under regular observation schemes.

Consequently,
\[
I_O(\theta;\Delta)
=
I_C(\theta)-I_M(\theta;\Delta)
\rightarrow
I_C(\theta).
\]

\begin{proposition}
If the complete-data information matrix $I_C(\theta)$ is positive definite and
\[
I_M(\theta;\Delta)\rightarrow0
\]
as $\Delta\rightarrow0$, then
\[
I_O(\theta;\Delta)
\rightarrow
I_C(\theta)
\]
and local identifiability under complete observation implies local identifiability under sufficiently frequent observation.
\end{proposition}

This result formalises the intuitive notion that increasing observation frequency progressively removes uncertainty about latent infection and network events.

\subsection{Dense versus Sparse Networks}

The rate at which information accumulates depends strongly on network topology. Let
\[
d_N
=
\frac{1}{N}
\sum_{i=1}^{N}
\sum_{j\neq i}
A_{ij}
\]
denote the average degree of contact.

In sparse-network asymptotics,
\[
d_N=O(1),
\]
so each individual has only a bounded number of contacts as $N\rightarrow\infty$.

In dense-network asymptotics,
\[
d_N=O(N).
\]

The amount of transmission information is driven by the number of susceptible--infectious contacts,
\[
E_{SI}(t)
=
\sum_{i,j}
A_{ij}(t)
\mathbf{1}\{X_i(t)=S,X_j(t)=I\}.
\]
For sparse networks,
\[
E_{SI}(t)=O(N),
\]
whereas for dense networks,
\[
E_{SI}(t)=O(N^2).
\]

Consequently, information about the transmission rate $\beta$ may accumulate substantially faster in dense networks. However, increased dependence between observations may also reduce the effective information per contact. The asymptotic scaling of
\[
I_O^{(N)}(\beta)
\]
therefore depends on both network density and epidemic dependence structure.

\subsection{Identifiability Phase Diagram}

The preceding results suggest that identifiability is governed by three principal factors:

\[
\mathcal{D}
=
(d_N,\Delta,\rho),
\]
where $d_N$ denotes network density, $\Delta$ denotes observation frequency, and $\rho$ denotes observation quality (e.g., sensitivity and specificity of symptom and contact measurements).

Define the asymptotic information index
\[
\Phi(\theta;\mathcal{D})
=
\lambda_{\min}
\left\{
I_O(\theta;\mathcal{D})
\right\}.
\]

We classify observation regimes as

\[
\Phi(\theta;\mathcal{D})>k
\]
(identifiable),

\[
0<\Phi(\theta;\mathcal{D})\le k
\]
(weakly identifiable),

and

\[
\Phi(\theta;\mathcal{D})=0
\]
(non-identifiable),

for some small threshold $k>0$.

The resulting phase diagram partitions the observation-design space into regions of strong, weak, and absent identifiability. Such diagrams provide a practical tool for determining the observation frequency, network coverage, and measurement quality required to reliably estimate epidemic transmission parameters.

More generally, the phase-diagram perspective highlights that identifiability is not solely a property of the epidemic model but rather an emergent property of the interaction between epidemic dynamics, network structure, and the observation process.

\section{Simulation Studies}
\label{sec7:sim}
The simulation study is designed to evaluate the theoretical results developed in Sections \ref{sec4:inf-struc}--\ref{sec6:asymp} and to investigate how information loss induced by partial observation affects parameter estimation and identifiability. Particular attention is given to the relationship between observation quality, network observability, and the finite-sample behaviour of epidemic transmission parameter estimates. The simulations provide an empirical validation of the information decomposition, identifiability conditions, and asymptotic results established earlier.

\subsection{Experimental Design}
Synthetic datasets are generated from the continuous-time SEIR dynamic-network model described in Section \ref{sec2:frame}. For each simulation replicate, a latent epidemic--network trajectory
\[
Z(t)=\{X(t),A(t)\}, \qquad 0\le t\le T,
\]
is first generated from the complete-data process and subsequently passed through the observation model to produce symptom and contact data. The baseline parameter vector
\[
\theta_0 = (\beta,\xi,\kappa,\gamma,\eta,\tau,p_E,p_I,s,c)
\]
is selected to produce realistic epidemic growth and network dynamics. For each scenario, $R=500$ independent datasets are simulated.

Three observation regimes are considered:
\begin{enumerate}
\item \textbf{High-information regime}: frequent symptom observations, high network coverage, and low measurement error.
\item \textbf{Moderate-information regime}: intermediate observation frequency and moderate contact misclassification.
\item \textbf{Sparse-information regime}: infrequent observations, substantial missing contact information, and increased measurement error.
\end{enumerate}
The regimes differ through observation frequency ($\Delta$), symptom and contact observation quality ($p_E,p_I,s,c$), and the level of external infection pressure ($\xi$). These factors directly influence the amount of information available for inference and therefore provide a natural framework for evaluating information loss and identifiability.

\subsection{Information Recovery Under Partial Observation}
The first set of simulations investigates the empirical behaviour of the observed Fisher information matrix under different observation regimes. For each simulated dataset, the observed information matrix is estimated from the observed-data likelihood and compared with the corresponding complete-data information matrix. Particular attention is given to the smallest eigenvalue
\[
\lambda_{\min}(I_O(\theta)),
\]
the condition number
\[
\phi_I = \frac{\lambda_{\max}(I_O(\theta))}{\lambda_{\min}(I_O(\theta))},
\]
and the relative information matrix
\[
R(\theta) = I_C(\theta)^{-1/2} I_O(\theta) I_C(\theta)^{-1/2}.
\]
The principal results are reported in Table~\ref{tab:info-recovery} and Figures~\ref{fig:param-recovery-coverage}. Table~\ref{tab:info-recovery} summarises the average complete-data information, observed-data information, relative information retained, smallest eigenvalue of the observed information matrix, and information condition number across the three observation regimes.

\begin{table}[ht]
\centering
\caption{Information recovery under different observation regimes.}
\label{tab:info-recovery}
\begin{tabular}{lccccc}
\hline
Regime & $\mathrm{tr}(I_C)$ & $\mathrm{tr}(I_O)$ & $\lambda_{\min}(I_O)$ & $\phi_I$ & $R(\theta)$ \\
\hline
High     & 645 & 582 & 44.1759 & 2.7738 &	0.9025 \\
Moderate & 645 & 362 & 30.7045	& 2.3179 & 0.5615 \\
Sparse   & 645 & 164 & 7.4907 & 5.3928 & 0.2536 \\
\hline
\end{tabular}
\end{table}

Figure~\ref{fig:param-recovery-coverage}(b) displays the distribution of $\lambda_{\min}(I_O)$ across Monte Carlo replicates, while Figure~\ref{fig:param-recovery-coverage}(c) illustrates the proportion of complete-data information retained under each observation regime. Together, these results provide a finite-sample assessment of the information decomposition developed in Section \ref{sec4:inf-struc} and quantify the extent to which information deteriorates as observation quality decreases.

\subsection{Parameter Recovery and Inferential Performance}
The second set of simulations examines the practical consequences of information loss for parameter estimation. For each replicate, the model parameters are estimated using the proposed inferential framework, and finite-sample performance is assessed through bias, root mean squared error (RMSE), empirical coverage probability, and average interval width.

The primary focus is on the epidemic transmission parameters
\[
(\beta,\xi,\kappa,\gamma,p_E,p_I,s,c),
\]
although all model parameters are evaluated. Estimation accuracy is assessed through
\[
\mathrm{Bias}(\hat{\theta}), \qquad \mathrm{RMSE}(\hat{\theta}), \qquad \widehat{\mathrm{CP}}, \qquad \widehat{W},
\]
representing bias, root mean squared error, coverage probability, and average interval width, respectively.

Results are summarised in Table~\ref{tab:param-recovery} and Figures~\ref{fig:param-recovery-coverage}. Table~\ref{tab:param-recovery} reports the average estimation performance across simulation replicates.

\begin{table}[ht]
\centering
\caption{Parameter recovery under different observation regimes.}
\label{tab:param-recovery}
\begin{tabular}{llcccccc}
\hline
Parameter	&	Regime	&	Truth	&	Mean estimate	&	Bias	&	RMSE	&	Coverage	&	Mean width	\\
\hline
$\beta$	&	High	&	0.30	&	0.3021	&	0.0021	&	0.0266	&	0.946	&	0.1032	\\
	&	Moderate	&	0.30	&	0.3156	&	0.0156	&	0.0383	&	0.930	&	0.1400	\\
	&	Sparse	&	0.30	&	0.3219	&	0.0219	&	0.0667	&	0.946	&	0.2450	\\ \hline
$\gamma$	&	High	&	0.25	&	0.2503	&	0.0003	&	0.0257	&	0.954	&	0.1032	\\
	&	Moderate	&	0.25	&	0.2542	&	0.0042	&	0.0319	&	0.954	&	0.1225	\\
	&	Sparse	&	0.25	&	0.2541	&	0.0041	&	0.0404	&	0.946	&	0.1633	\\ \hline
$\kappa$	&	High	&	0.40	&	0.3997	&	-0,0003	&	0.0328	&	0.94	&	0.1238	\\
	&	Moderate	&	0.40	&	0.4014	&	0.0014	&	0.0380	&	0.956	&	0.1470	\\
	&	Sparse	&	0.40	&	0.4007	&	0.0007	&	0.0495	&	0.956	&	0.1960	\\ \hline
$\xi$	&	High	&	0.03	&	0.0312	&	0.0012	&	0.0109	&	0.946	&	0.0413	\\
	&	Moderate	&	0.03	&	0.0346	&	0.0046	&	0.0142	&	0.920	&	0.0524	\\
	&	Sparse	&	0.03	&	0.0381	&	0.0081	&	0.0226	&	0.936	&	0.0800	\\ \hline
$p_E$	&	High	&	0.35	&	0.3469	&	-0.0031	&	0.0367	&	0.944	&	0.1444	\\
	&	Moderate	&	0.35	&	0.3441	&	-0.0059	&	0.0449	&	0.942	&	0.1715	\\
	&	Sparse	&	0.35	&	0.3453	&	-0.0047	&	0.0584	&	0.952	&	0.2287	\\ \hline
$p_I$	&	High	&	0.80	&	0.8006	&	0.0006	&	0.0300	&	0.970	&	0.1238	\\
	&	Moderate	&	0.80	&	0.7972	&	-0.0028	&	0.0367	&	0.960	&	0.1470	\\
	&	Sparse	&	0.80	&	0.7964	&	-0.0036	&	0.0489	&	0.956	&	0.1958	\\ \hline
$s$	&	High	&	0.85	&	0.8510	&	0.0010	&	0.0365	&	0.944	&	0.1443	\\
	&	Moderate	&	0.85	&	0.8440	&	-0.006	&	0.0488	&	0.956	&	0.1931	\\
	&	Sparse	&	0.85	&	0.8407	&	-0.0093	&	0.0828	&	0.976	&	0.3021	\\ \hline
$c$	&	High	&	0.90	&	0.8994	&	-0.0006	&	0.0321	&	0.944	&	0.1219	\\ 
	&	Moderate	&	0.90	&	0.8969	&	-0.0031	&	0.0432	&	0.950	&	0.1580	\\
	&	Sparse	&	0.90	&	0.8807	&	-0.0193	&	0.0745	&	0.960	&	0.2476	\\

\hline
\end{tabular}
\end{table}

\begin{figure}[p]
\centering

\begin{minipage}[b]{0.48\textwidth}
    \includegraphics[width=\linewidth]{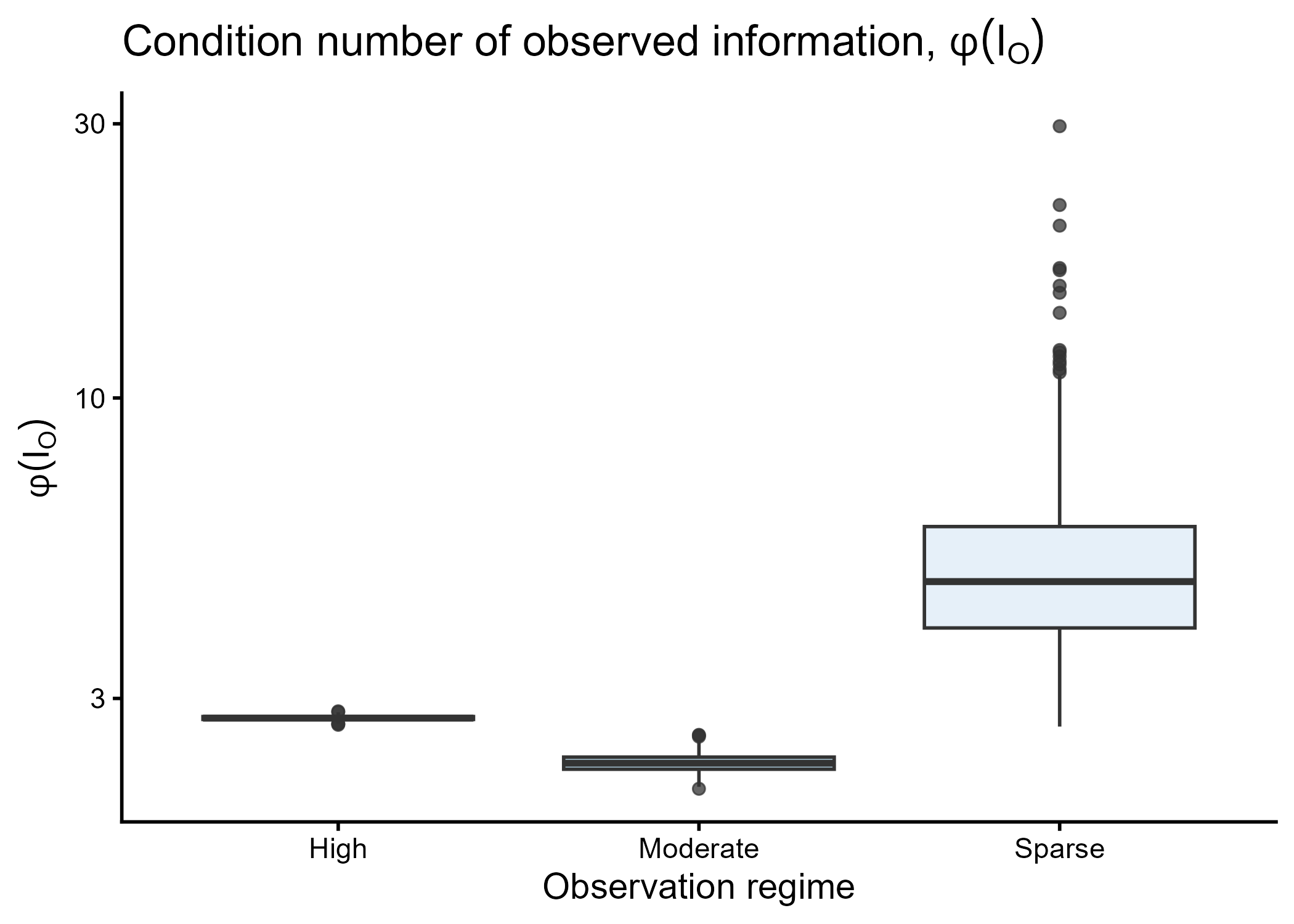}       
    \caption*{(a) Condition number, $\phi(I_O)$}
    \label{fig:condition-num}
\end{minipage}
\hfill
\begin{minipage}[b]{0.48\textwidth}
    \includegraphics[width=\linewidth]{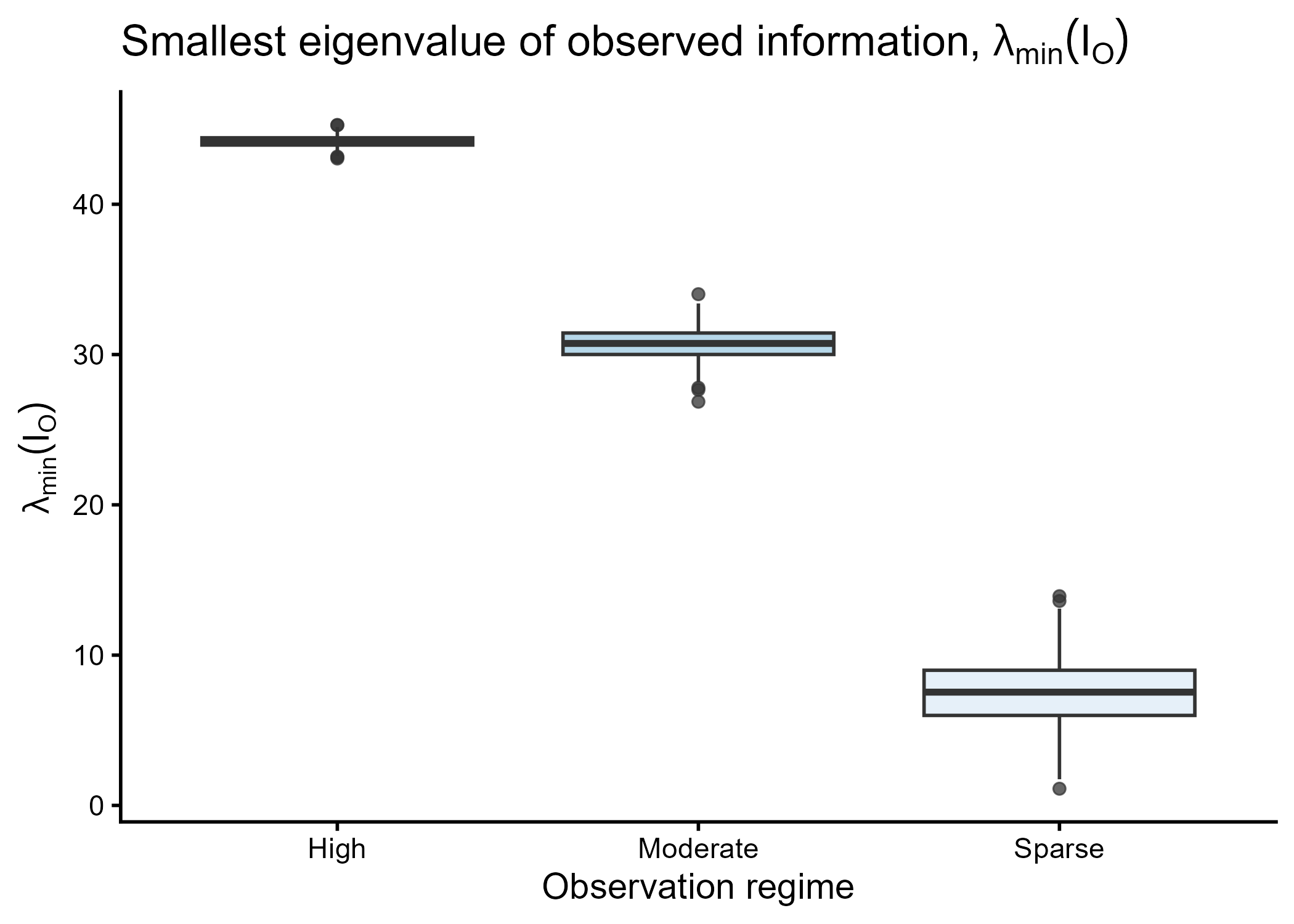}
    \caption*{(b) Smallest eigenvalue, $\lambda_{\text{min}}(I_O)$}
    \label{fig:lambda-min}
\end{minipage}

\vspace{0.5cm}

\begin{minipage}[b]{0.48\textwidth}
    \includegraphics[width=\linewidth]{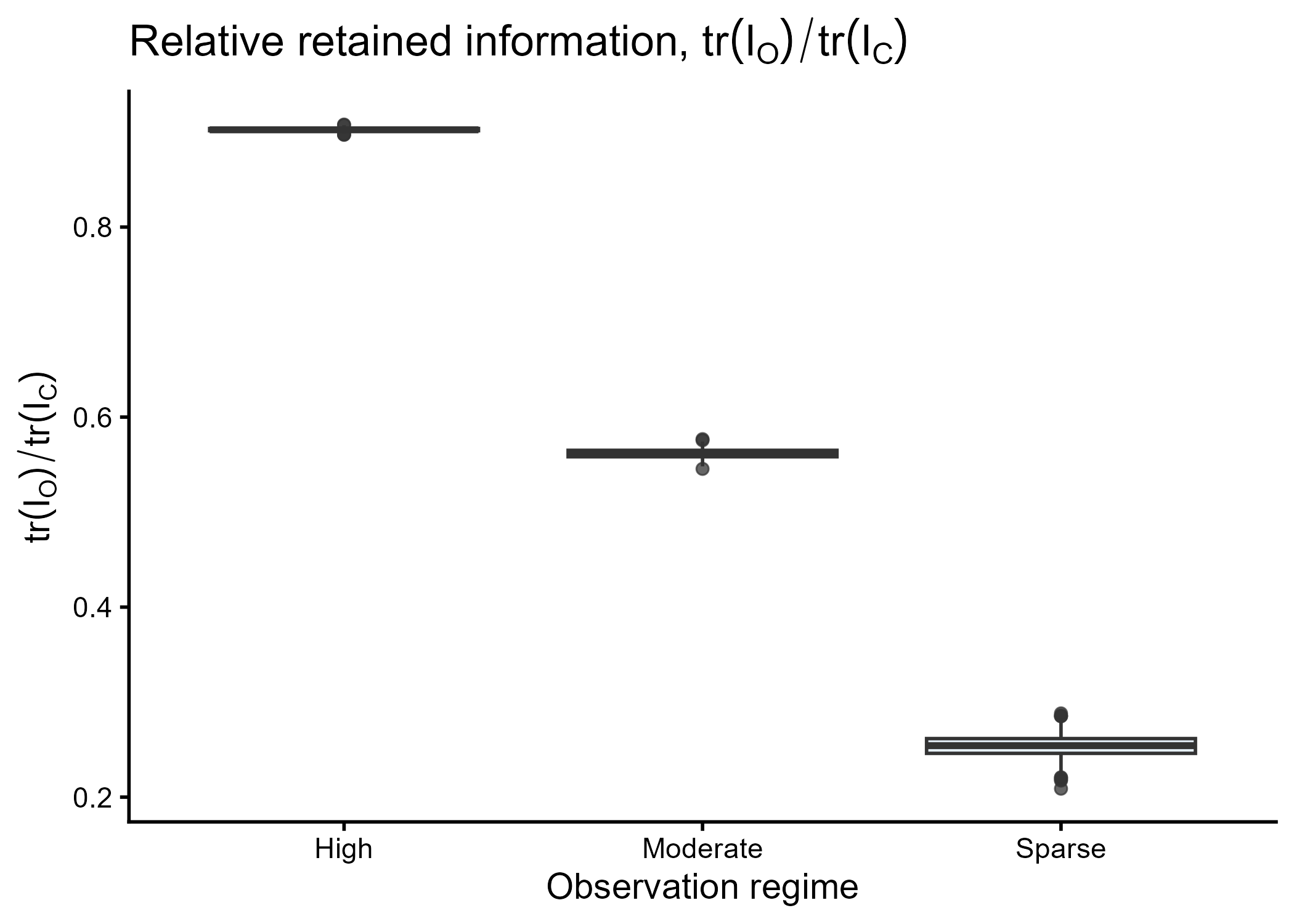}
    \caption*{(c) Relative retained information $tr(I_O)/tr(I_C)$}
    \label{fig:relative-info}
\end{minipage}
\hfill
\begin{minipage}[b]{0.48\textwidth}
    \includegraphics[width=\linewidth]{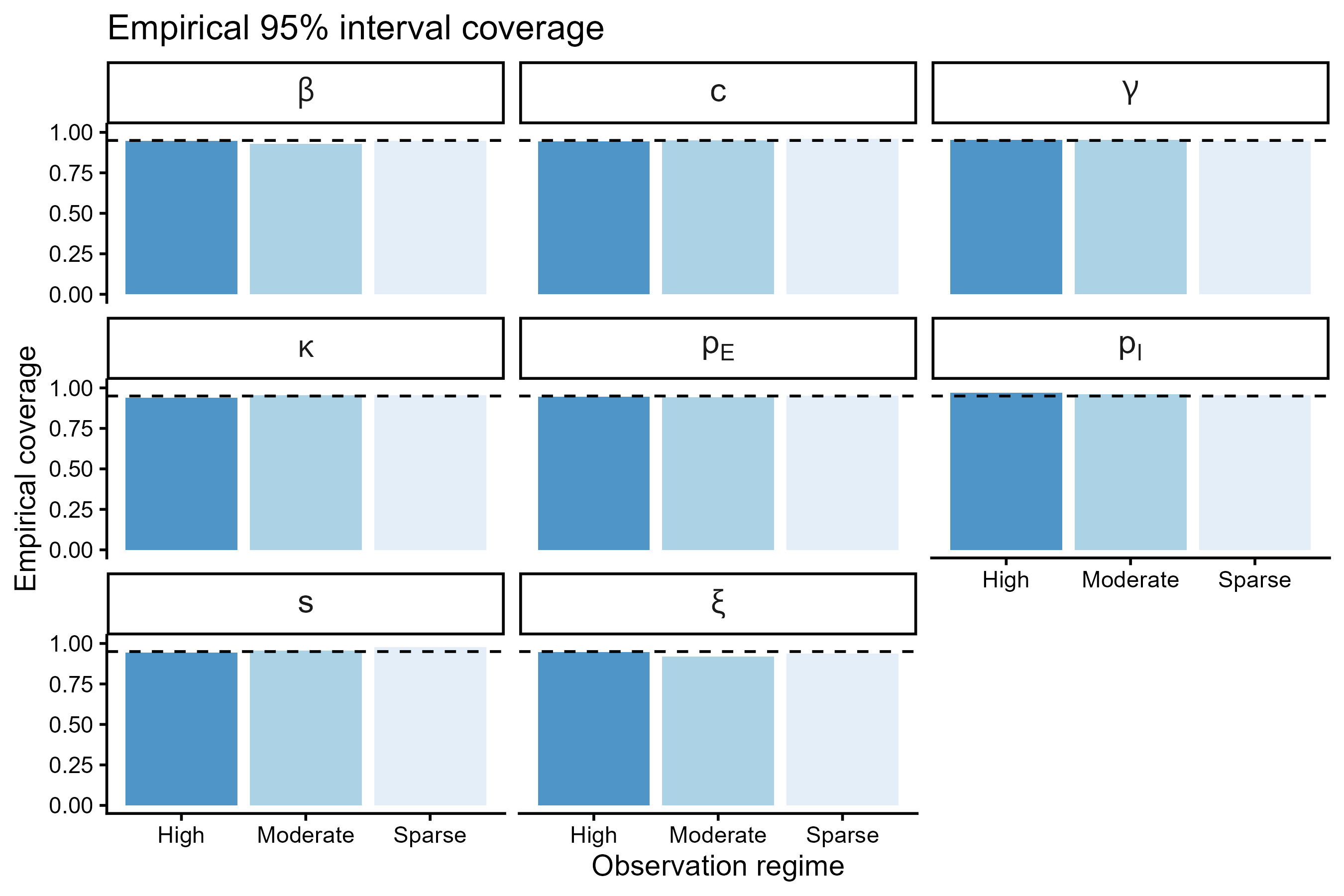}
    \caption*{(d) Coverage by regime}
    \label{fig:cov-reg}
\end{minipage}

\vspace{0.5cm}

\begin{minipage}[b]{0.48\textwidth}
    \includegraphics[width=\linewidth]{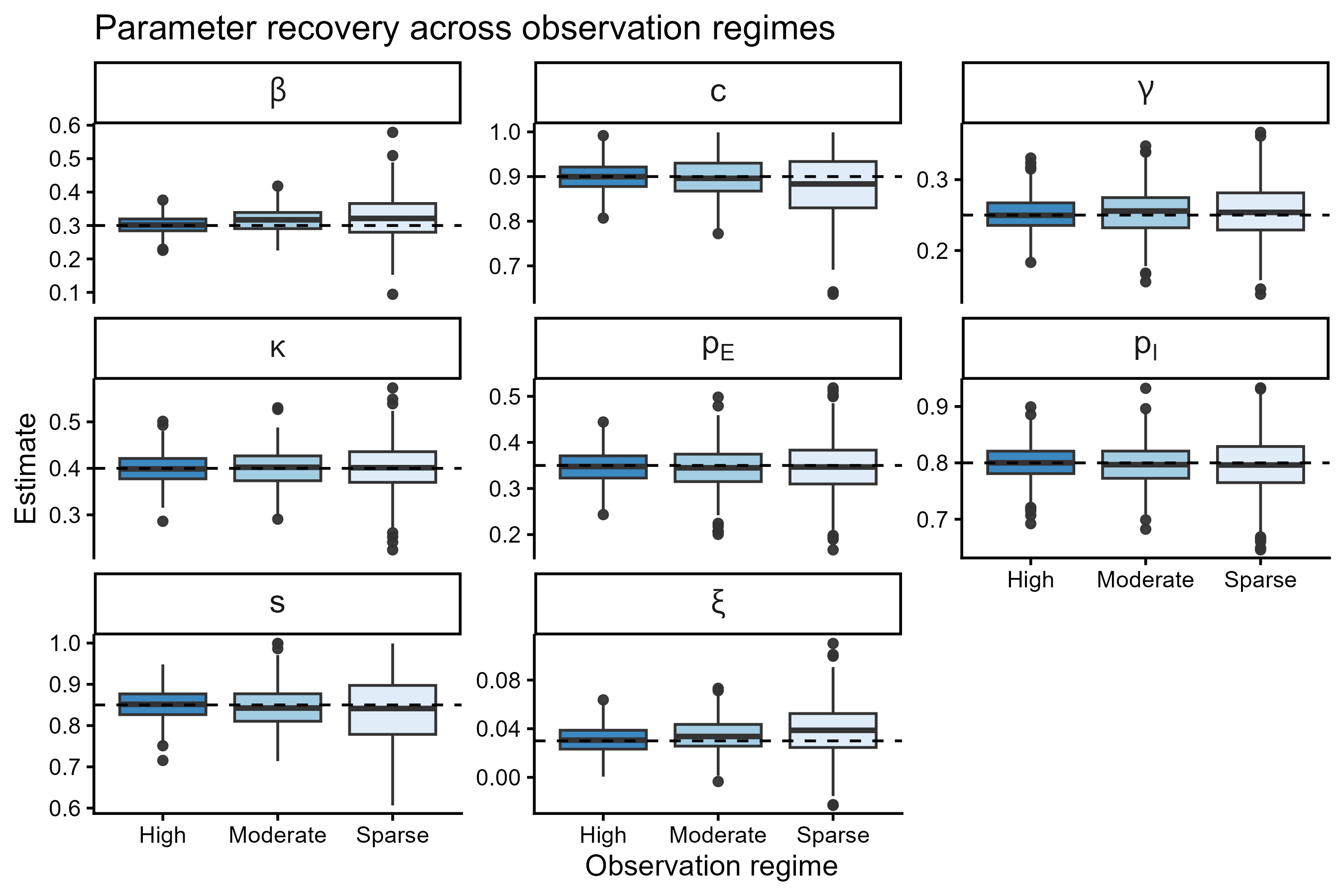}
    \caption*{(e) Parameter recovery}
    \label{fig:par-recov}
\end{minipage}
\hfill
\begin{minipage}[b]{0.48\textwidth}
    \includegraphics[width=\linewidth]{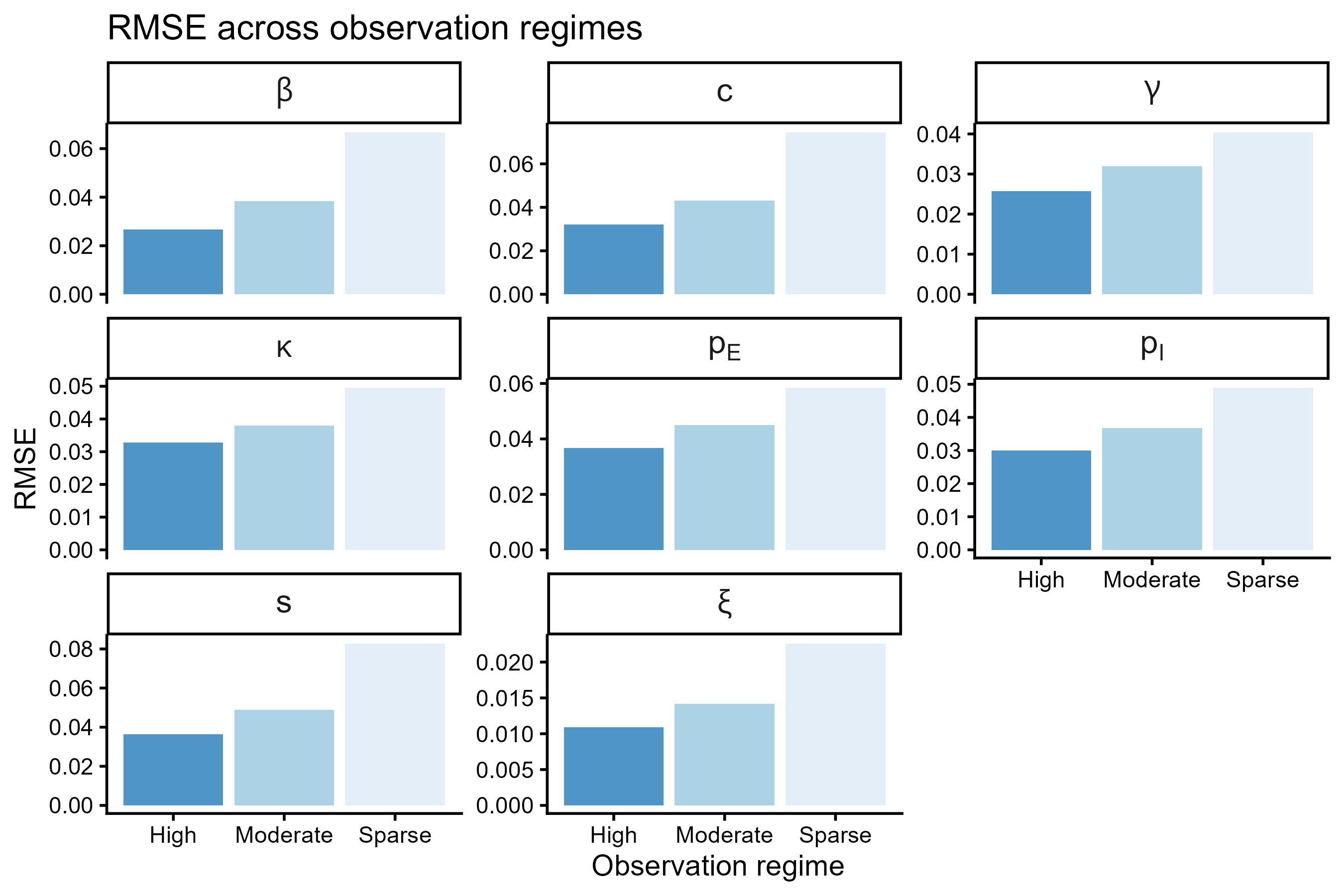}
    \caption*{(f) RMSE by regime}
    \label{fig:rmse-reg}
\end{minipage}

\caption{Recovery and coverage of information and true parameter values}
\label{fig:param-recovery-coverage}
\end{figure}

Figure~\ref{fig:param-recovery-coverage} (a) to (c) presents the distribution of parameter estimates across observation regimes, while (d) to (f) summarises RMSE and empirical coverage probabilities. These results allow direct assessment of whether reductions in observed information translate into deterioration of finite-sample inferential performance.

\subsection{Empirical Identifiability Boundaries}
The final simulation component investigates the existence of empirical transition regions separating identifiable, weakly identifiable, and non-identifiable observation regimes. The theoretical developments of Sections \ref{sec5:inf-loss} and \ref{sec6:asymp} predict that these boundaries should arise as a consequence of declining observed information and increasing parameter confounding.

Simulations are therefore performed over a grid of observation frequencies, network coverages, and measurement accuracies. For each design point
\[
\mathcal{D}(d_N,\Delta,\rho),
\]
the quantities
\[
\lambda_{\min}(I_O), \qquad \phi_I, \qquad \mathrm{RMSE}, \qquad \widehat{\mathrm{CP}}
\]
are evaluated and mapped across the observation-design space.

Particular attention is given to the transmission parameter $\beta$ and the external infection parameter $\xi$, whose identifiability was shown in Section \ref{sec5:inf-loss} to depend strongly on the availability of network information. Increasing posterior dependence between these parameters is interpreted as evidence of approaching a transmission--external infection identifiability boundary.

The resulting empirical phase diagrams are reported in Figures~\ref{fig:param-recovery-coverage}, where (b) is based on the smallest eigenvalue of the observed information matrix, while the (f) is based on RMSE. These diagrams partition the observation-design space into regions of strong identifiability, weak identifiability, and practical non-identifiability. Comparison with the theoretical phase boundaries provides a direct validation of the information-theoretic framework developed in this paper.

\subsection{Summary of Simulation Findings}
Across all simulation scenarios, empirical information measures closely track finite-sample estimation performance. As observation quality decreases, the observed Fisher information matrix becomes increasingly ill-conditioned, leading to larger estimation errors, wider uncertainty intervals, and stronger parameter confounding. The most deterioration is expected for the transmission parameter $\beta$ and the external infection parameter $\xi$, consistent with the theoretical confounding results established in Section \ref{sec5:inf-loss}. Overall, the simulation results provide strong empirical support for the proposed information framework and demonstrate that information-theoretic quantities successfully characterise the transition between identifiable and weakly identifiable observation regimes.

\section{Discussion and Conclusion}
\label{sec8:con}

Inference for infectious disease transmission on dynamic contact networks is fundamentally constrained by latent infection times, incomplete network observations, measurement error, and external sources of infection. Although these challenges are ubiquitous in modern epidemic studies, they are often treated separately or addressed through simplifying assumptions whose implications for identifiability remain unclear. This paper develops a unified statistical framework for understanding identifiability and information in epidemic transmission models on partially observed dynamic networks.

The proposed framework treats epidemic progression, network evolution, and observation processes as components of a single partially observed stochastic system. By defining identifiability through the observed-data law and linking it to the geometry of the observed Fisher information matrix, the framework provides a rigorous basis for determining which parameters can be estimated reliably under a given observation regime. The information decomposition developed in this paper clarifies how latent infection times, missing contact histories, and imperfect observations contribute to information loss and parameter confounding. In particular, the analysis formally characterises conditions under which transmission and external infection processes become observationally indistinguishable.

Beyond establishing structural and local identifiability conditions, the paper introduces an information-theoretic perspective on epidemic inference. The concepts of relative information, missing information, and identifiability phase boundaries provide quantitative tools for assessing the inferential consequences of alternative observation regimes. The simulation studies demonstrate that these theoretical quantities closely track finite-sample estimation performance and successfully identify transitions between identifiable and weakly identifiable regimes.

The framework also has important implications for epidemic study design. By expressing observation strategies through their impact on the observed information matrix, it becomes possible to evaluate the relative value of symptom surveillance, contact-network measurement, and external exposure assessment before data collection begins. This provides a principled foundation for designing surveillance systems that maximise inferential efficiency under practical resource constraints.

Although motivated by infectious disease transmission, the proposed methodology applies more broadly to interacting stochastic processes evolving on partially observed networks. Similar challenges arise in information diffusion, behavioural contagion, ecological interaction systems, and other complex networked processes. Future work will focus on deriving sharper asymptotic identifiability results, developing optimal information-based observation designs, and extending the framework to high-dimensional and adaptive network settings. More generally, the results demonstrate that identifiability is not solely a property of the epidemic mechanism itself but emerges from the interaction between latent dynamics, network structure, and the observation process. Understanding this interaction is essential for reliable inference in modern epidemic systems and other partially observed stochastic networks.

\section*{Acknowledgements}
The author thanks colleagues and collaborators for helpful discussions.

\bibliographystyle{apalike}
\bibliography{IDENINFreferences}

\begin{thebibliography}{}

\bibitem[Abed et~al., 2026]{abed2026spatial}
Abed, A., Torabi, M., and Mashreghi, Z. (2026).
\newblock Spatial individual-level models for transmission dynamics of seasonal
  infectious diseases.
\newblock {\em Statistics in Medicine}, 45(3-5):e70384.

\bibitem[Almutiry and Deardon, 2021]{almutiry2021contact}
Almutiry, W. and Deardon, R. (2021).
\newblock Contact network uncertainty in individual level models of infectious
  disease transmission.
\newblock {\em Statistical Communications in Infectious Diseases},
  13(1):20190012.

\bibitem[Asaduzzaman, 2026]{asaduzzaman2026complete}
Asaduzzaman, M. (2026).
\newblock A complete-data likelihood for epidemic processes on partially
  observed dynamic networks.
\newblock {\em arXiv preprint arXiv:2607.15179}.

\bibitem[Bansal et~al., 2010]{bansal2010dynamic}
Bansal, S., Read, J., Pourbohloul, B., and Meyers, L.~A. (2010).
\newblock The dynamic nature of contact networks in infectious disease
  epidemiology.
\newblock {\em Journal of Biological Dynamics}, 4(5):478--489.

\bibitem[Becker and Britton, 1999]{becker1999statistical}
Becker, N.~G. and Britton, T. (1999).
\newblock Statistical studies of infectious disease incidence.
\newblock {\em Journal of the Royal Statistical Society: Series B (Statistical
  Methodology)}, 61(2):287--307.

\bibitem[Black, 2019]{black2019importance}
Black, A.~J. (2019).
\newblock Importance sampling for partially observed temporal epidemic models.
\newblock {\em Statistics and Computing}, 29(4):617--630.

\bibitem[Bret{\'o}, 2018]{breto2018modeling}
Bret{\'o}, C. (2018).
\newblock Modeling and inference for infectious disease dynamics: a
  likelihood-based approach.
\newblock {\em Statistical Science: A Review Journal of the Institute of
  Mathematical Statistics}, 33(1):57.

\bibitem[Browning et~al., 2022]{browning2022efficient}
Browning, A.~P., Drovandi, C., Turner, I.~W., Jenner, A.~L., and Simpson, M.~J.
  (2022).
\newblock Efficient inference and identifiability analysis for differential
  equation models with random parameters.
\newblock {\em PLOS Computational Biology}, 18(11):e1010734.

\bibitem[Bu et~al., 2025]{bu2025stochastic}
Bu, F., Aiello, A.~E., Volfovsky, A., and Xu, J. (2025).
\newblock Stochastic em algorithm for partially observed stochastic epidemics
  with individual heterogeneity.
\newblock {\em Biostatistics}, 26(1):kxae018.

\bibitem[Bu et~al., 2022]{bu2022likelihood}
Bu, F., Aiello, A.~E., Xu, J., and Volfovsky, A. (2022).
\newblock Likelihood-based inference for partially observed epidemics on
  dynamic networks.
\newblock {\em Journal of the American Statistical Association},
  117(537):510--526.

\bibitem[Daley and Gani, 1999]{daley1999epidemic}
Daley, D.~J. and Gani, J.~M. (1999).
\newblock {\em Epidemic modelling: an introduction}.
\newblock Number~15. Cambridge University Press.

\bibitem[Danon et~al., 2011]{danon2011networks}
Danon, L., Ford, A.~P., House, T., Jewell, C.~P., Keeling, M.~J., Roberts,
  G.~O., Ross, J.~V., and Vernon, M.~C. (2011).
\newblock Networks and the epidemiology of infectious disease.
\newblock {\em Interdisciplinary Perspectives on Infectious Diseases},
  2011(1):284909.

\bibitem[Eames et~al., 2015]{eames2015six}
Eames, K., Bansal, S., Frost, S., and Riley, S. (2015).
\newblock Six challenges in measuring contact networks for use in modelling.
\newblock {\em Epidemics}, 10:72--77.

\bibitem[Fintzi et~al., 2017]{fintzi2017efficient}
Fintzi, J., Cui, X., Wakefield, J., and Minin, V.~N. (2017).
\newblock Efficient data augmentation for fitting stochastic epidemic models to
  prevalence data.
\newblock {\em Journal of Computational and Graphical Statistics},
  26(4):918--929.

\bibitem[Groendyke et~al., 2011]{groendyke2011bayesian}
Groendyke, C., Welch, D., and Hunter, D.~R. (2011).
\newblock Bayesian inference for contact networks given epidemic data.
\newblock {\em Scandinavian Journal of Statistics}, 38(3):600--616.

\bibitem[Heesterbeek et~al., 2015]{heesterbeek2015modeling}
Heesterbeek, H., Anderson, R.~M., Andreasen, V., Bansal, S., De~Angelis, D.,
  Dye, C., Eames, K.~T., Edmunds, W.~J., Frost, S.~D., Funk, S., et~al. (2015).
\newblock Modeling infectious disease dynamics in the complex landscape of
  global health.
\newblock {\em Science}, 347(6227):aaa4339.

\bibitem[Hethcote, 2000]{hethcote2000mathematics}
Hethcote, H.~W. (2000).
\newblock The mathematics of infectious diseases.
\newblock {\em SIAM Review}, 42(4):599--653.

\bibitem[Huang et~al., 2024]{huang2024detecting}
Huang, J., Morsomme, R., Dunson, D., and Xu, J. (2024).
\newblock Detecting changes in the transmission rate of a stochastic epidemic
  model.
\newblock {\em Statistics in Medicine}, 43(10):1867--1882.

\bibitem[Istvan et~al., 2019]{istvan2019mathematics}
Istvan, Z., MILLER, K., JOEL, C.~S., and Peter, L. (2019).
\newblock {\em Mathematics of Epidemics on Networks: From Exact to Approximate
  Models}.
\newblock Springer.

\bibitem[Kamkumo et~al., 2025]{kamkumo2025estimating}
Kamkumo, F.~O., Njiasse, I.~M., and Wunderlich, R. (2025).
\newblock Estimating unobservable states in stochastic epidemic models with
  partial information.
\newblock {\em arXiv preprint arXiv:2506.00906}.

\bibitem[Lam et~al., 2022]{lam2022practical}
Lam, N.~N., Docherty, P.~D., and Murray, R. (2022).
\newblock Practical identifiability of parametrised models: A review of
  benefits and limitations of various approaches.
\newblock {\em Mathematics and Computers in Simulation}, 199:202--216.

\bibitem[Louis, 1982]{louis1982finding}
Louis, T.~A. (1982).
\newblock Finding the observed information matrix when using the em algorithm.
\newblock {\em Journal of the Royal Statistical Society Series B: Statistical
  Methodology}, 44(2):226--233.

\bibitem[Morsomme and Xu, 2025]{morsomme2025exact}
Morsomme, R. and Xu, J. (2025).
\newblock Exact bayesian inference for fitting stochastic epidemic models to
  partially observed incidence data.
\newblock {\em The Annals of Applied Statistics}, 19(3):2279--2293.

\bibitem[O’Neill and Roberts, 1999]{o1999bayesian}
O’Neill, P.~D. and Roberts, G.~O. (1999).
\newblock Bayesian inference for partially observed stochastic epidemics.
\newblock {\em Journal of the Royal Statistical Society Series A: Statistics in
  Society}, 162(1):121--129.

\bibitem[Pastor-Satorras et~al., 2015]{pastor2015epidemic}
Pastor-Satorras, R., Castellano, C., Van~Mieghem, P., and Vespignani, A.
  (2015).
\newblock Epidemic processes in complex networks.
\newblock {\em Reviews of Modern Physics}, 87(3):925--979.

\bibitem[Pellis et~al., 2015]{pellis2015eight}
Pellis, L., Ball, F., Bansal, S., Eames, K., House, T., Isham, V., and Trapman,
  P. (2015).
\newblock Eight challenges for network epidemic models.
\newblock {\em Epidemics}, 10:58--62.

\bibitem[Preston et~al., 2025]{preston2025think}
Preston, S.~P., Wilkinson, R.~D., Clayton, R.~H., Chappell, M.~J., and Mirams,
  G.~R. (2025).
\newblock Think before you fit: parameter identifiability, sensitivity and
  uncertainty in systems biology models.
\newblock {\em Current Opinion in Systems Biology}, page 100563.

\bibitem[Reeves et~al., 2024]{reeves2024model}
Reeves, S.~W., Lubold, S., Chandrasekhar, A.~G., and McCormick, T.~H. (2024).
\newblock Model-based inference and experimental design for interference using
  partial network data.
\newblock {\em arXiv preprint arXiv:2406.11940}.

\bibitem[Saucedo et~al., 2024]{saucedo2024comparative}
Saucedo, O., Laubmeier, A., Tang, T., Levy, B., Asik, L., Pollington, T., and
  Prosper, O. (2024).
\newblock Comparative analysis of practical identifiability methods for an seir
  model.
\newblock {\em arXiv preprint arXiv:2401.15076}.

\bibitem[T{\"o}nsing et~al., 2018]{tonsing2018profile}
T{\"o}nsing, C., Timmer, J., and Kreutz, C. (2018).
\newblock Profile likelihood-based analyses of infectious disease models.
\newblock {\em Statistical Methods in Medical Research}, 27(7):1979--1998.

\bibitem[Vecherin et~al., 2026]{vecherin2026infection}
Vecherin, S.~N., Meyer, A.~C., Cummings, C.~L., Trump, B.~D., Ehlschlaeger,
  C.~R., and Linkov, I. (2026).
\newblock Infection risk assessment for socially structured population using
  stochastic microexposure model.
\newblock {\em Journal of Exposure Science \& Environmental Epidemiology},
  36(2):386--397.

\bibitem[Volz and Meyers, 2007]{volz2007susceptible}
Volz, E. and Meyers, L.~A. (2007).
\newblock Susceptible--infected--recovered epidemics in dynamic contact
  networks.
\newblock {\em Proceedings of the Royal Society B: Biological Sciences},
  274(1628):2925.

\bibitem[Yang et~al., 2012]{yang2012hybrid}
Yang, Y., Longini~Jr, I.~M., Halloran, M.~E., and Obenchain, V. (2012).
\newblock A hybrid {EM} and {M}onte {C}arlo {EM} algorithm and its application
  to analysis of transmission of infectious diseases.
\newblock {\em Biometrics}, 68(4):1238--1249.

\end{thebibliography}

\end{document}